\documentclass{amsart}
\usepackage{amssymb}
\usepackage{enumerate}
\usepackage{url}
\usepackage[all]{xy}
\usepackage{amsmath,amsfonts,amssymb}
\usepackage{ulem}
\usepackage{color}
\usepackage{cleveref}
\usepackage{faktor}
\usepackage{mathtools}
\usepackage{verbatim}

\newcommand\Z{\mathbb Z}

\newcommand{\cha}{\operatorname{char}}

\newcommand{\gal}{\operatorname{Gal}}
\newcommand{\im}{\operatorname{im}}

\newcommand\ph\varphi
\newcommand\ps\psi
\newcommand\ep\varepsilon
\newcommand\rh\varrho
\newcommand\al\alpha
\newcommand\be\beta
\newcommand\ga\gamma
\newcommand\om\omega
\newcommand\ta\tau
\renewcommand\th\theta
\newcommand\de\delta
\newcommand\ze\zeta
\newcommand\ch\chi
\newcommand\et\eta
\newcommand\io\iota
\newcommand\la\lambda
\newcommand\si\sigma

\newcommand\Ga\Gamma
\newcommand\De\Delta
\newcommand\Th\Theta
\newcommand\La\Lambda
\newcommand\Si\Sigma
\newcommand\Ph\Phi
\newcommand\Ps\Psi
\newcommand\Om\Omega

\newtheorem{theorem}{Theorem}[section]
\newtheorem{lemma}[theorem]{Lemma}
\newtheorem{proposition}[theorem]{Proposition}
\newtheorem{corollary}[theorem]{Corollary}

\theoremstyle{definition}

\newtheorem{example}[theorem]{Example}
\theoremstyle{remark}

\begin{document}
\title{Hereditarily non-pythagorean Fields}
\author{David Grimm and David B. Leep}
\address{David Grimm, Departamento de Matem\'atica y Ciencia de la Computaci\'on, Universidad de Santiago de Chile,  Santiago de Chile, Chile}
\email{david.grimm@usach.cl}
\address{David B. Leep\\
Department of Mathematics\\ University of Kentucky\\ Lexington, KY 40506-0027 U.S.A. } \email{leep@uky.edu}

\begin{abstract}
We prove for a large class of fields $F$ that every proper finite extension of 
$F_{pyth}$, the pythagorean closure of $F$, is not a pythagorean field.  
This class of fields contains number fields and fields $F$ that are 
finitely generated of transcendence degree at least one over some subfield of $F$.   
\end{abstract}


\maketitle

\section{Introduction}
A field  is called \textit{pythagorean} if every sum of squares is a square, and is called \textit{quadratically closed} if every element is a square.
In a nonreal field of characteristic different from $2$, every element is a sum of squares, and thus
such a field is quadratically closed if and only if it is pythagorean. Trivially, every field of characteristic $2$ is pythagorean.

For a field  $E$ we let  $E_\mathrm{pyth}$ denote the direct limit of all finite extensions of $E$ that consist of chains of quadratic extensions obtained by iteratively adjoining square roots (inside a fixed algebraic closure of $E$) of sums of two squares in the predecessor. 
Similarly, we let $E_\mathrm{quad}$ denote the direct limit of finite chains of all quadratic extensions obtained by iteratively adjoining square roots of arbitrary elements in the predecessor.  The field $E_\mathrm{pyth}$ is called the \textit{pythagorean closure} of $E$, and is a pythagorean field. The field $E_\mathrm{quad}$ is called the \textit{quadratic closure} of $E$, and is a quadratically closed field. In particular, $E_\mathrm{quad}$ is always nonreal.
If $E$ is nonreal and of characteristic different from $2$, then $E_\mathrm{pyth} = E_\mathrm{quad}$.

We say that a field $k$ is \textit{hereditarily non-pythagorean} if every proper 
finite extension $F$, with $F/k$ not purely inseparable, is non-pythagorean.  
By \cite[Chapter VIII, Theorem 5.7]{La-05}, every
non-pythagorean field is hereditarily non-pythagorean. The notion is thus only 
interesting for pythagorean fields of characteristic different from $2$. Similarly, a field $k$  is 
\textit{hereditarily non-quadratically closed} if every proper finite extension $F$, 
with $F/k$ not purely inseparable, is not quadratically closed.
By \cite[Chapter VIII, Corollary 5.11]{La-05} and its proof, every nonreal, 
non-quadratically closed field of characteristic different from $2$ is hereditarily non-quadratically closed. Hence for nonreal fields of characteristic different from $2$ the 
notion is interesting only for quadratically closed fields. A field of characteristic $2$ is quadratically closed if and only if it is perfect. Hence, every quadratically closed field of characteristic $2$ is hereditarily quadratically closed.

 If a field $F$ is hereditarily non-pythagorean, then it is hereditarily non-quadratically closed. The converse does not hold. For example,
 $\mathbb{R}(\!(t)\!)$ is hereditarily non-quadratically closed,
but it is not hereditarily non-pythagorean. In fact, $\mathbb{R}(\!(t)\!)$ is hereditarily pythagorean. Recall  from \cite{B-78} that a field $F$ is hereditarily pythagorean if $F$ is real pythagorean, and every finite real extension is also pythagorean.

In the definition of hereditarily non-pythagorean fields, we exclude the case where $F/k$ is purely inseparable, since if $k$ is nonreal pythagorean (or quadratically closed of characteristic $2$) and $F/k$ is  a finite 
purely inseparable extension, then $F$ is also  pythagorean (or quadratically closed of characteristic $2$). This is trivial in the case of characteristic $2$ and follows from \Cref{inseparable} otherwise.

The main results of this paper (\Cref{main theorem} and \Cref{main theorem II})
are that if $E$ is either an algebraic number field 
(a finite extension of $\mathbb{Q}$) or a field of characteristic different from $2$ that is finitely generated 
of transcendence degree at least one over some subfield, then 
$E_\mathrm{pyth}$ is hereditarily non-pythagorean.  

Previously, a version of this problem was studied for $E_\mathrm{quad}$ when $E$ is an algebraic number field. It was shown in
\cite[Chapter VII, Corollary 7.11]{La-05} that every proper finite extension of $E_\mathrm{quad}$ has infinitely many square classes. In particular, $E_\mathrm{quad}$ is hereditarily non-quadratically closed.

We further discuss the cases when $E$ is an infinite dimensional algebraic 
extension of a number field or an infinite dimensional algebraic extension of 
a function field of characteristic different from $2$. We prove in \Cref{new} 
that the same result as above holds under the additional assumption that $E$ is 
a Galois extension over some number field or over some function field of characteristic different from $2$.
In \Cref{example}, we construct a simple example of a nonreal infinite number 
field $E$ that is not pythagorean and not Galois over some number field, 
and such that 
every finite extension of $E_\mathrm{pyth}$ is  pythagorean. 
We prove some results in Section 7 that give a construction of a real 
infinite number field whose
pythagorean closure is not hereditarily non-pythagorean.

Hereditary properties of pythagorean closures and quadratic closures have been 
studied for other fields $E$.
In \cite[Chapter VII, Theorems 7.17, 7.18]{La-05}, it was shown that if $E$ is a local number field, then $E_\mathrm{quad}$ is hereditarily quadratically closed if $E$ is nondyadic, and is hereditarily non-quadratically closed  if $E$ is dyadic.
Becker wrote an extensive treatment of hereditarily pythagorean fields in \cite{B-78}. 

We let $\mathbb{Z}$ and $\mathbb{Q}$ denote the ring of integers and the field of rational numbers, respectively.  
A number field is a finite algebraic extension of $\mathbb{Q}$.  
An infinite number field will mean an infinite algebraic extension of $\mathbb{Q}$. 
A local field is a completion of a number field with respect to some nonarchimedean absolute value.  
A local field is dyadic if its residue field has characteristic $2$, and is nondyadic if its residue field has characteristic different from $2$.  
For a prime number $p$, we let $\mathbb{Q}_p$ denote the completion of 
$\mathbb{Q}$ with respect to the $p$-adic absolute value.
 
For a field $K$, we let $K^{\times} = K \backslash \{0\}$ and we let
$\cha K$ denote the characteristic of $K$.  
We let $K^2$ denote the set of squares of elements of $K$, and we
let $\sum K^2$ denote the set of sums of squares of elements in $K$. 
Let $(\sum K^2)^{\times} = (\sum K^2) \backslash \{0\}$.

The pythagoras number of a field $K$, written $p(K)$, is defined as 
the smallest integer $n$ such that each element in $\sum K^2$ can be written
as a sum of $n$ squares of elements in $K$.  If no such integer exists, we set
$p(K) = \infty$.

A field $K$ is nonreal if $-1 \in \sum K^2$.  Otherwise, a field is real.
A field $K$ is real if and only if $K$ has an ordering by 
\cite[Chapter VIII, Theorem 1.10]{La-05}.  
If $K$ is nonreal and $\cha K \ne 2$, then $K = \sum K^2$.
A field $K$ is pythagorean if $\sum K^2 = K^2$, and is quadratically closed if
$K = K^2$. Thus if $K$ is a pythagorean field, then $p(K) = 1$.

The topic of this paper fits into the broader context of studying the exact value of the pythagoras number and how the pythagoras number behaves under field extensions. 
Pfister studied the relation of the pythagoras number of a field $K$ and the pythagoras number of the rational function field $K(X)$ ( \cite[Chapter XI, Theorem 5.6 and Examples 5.9 (3)]{La-05} ). For $K= \mathbb{Q}_\mathrm{pyth}$, this yields 
$3 \leq p(\mathbb{Q}_\mathrm{pyth}(X))\leq 4$.
The pythagoras numbers of  hyperelliptic function fields over real fields were studied in \cite{BV-09}. For additional results on the behaviour of pythagoras numbers, see \cite[Chapter 7]{Pf-95}.

We use \cite{La-05} as a standard reference for other undefined terms and 
standard results.

\section{Some general results}

In this section we establish some general results and strategies that allow us to 
prove that certain finite extensions of pythagorean closures are not pythagorean.  One of the tools we use is valuation theory.
We will say that $v$ is a nonreal valuation on a field $K$ if the residue field of 
$(K,v)$ is a nonreal field.

\begin{lemma}\label{nonreal valuation}
Let $K$ be a field of characteristic different from $2$.  Let $v$ be a nonreal 
discrete valuation on $K$ with value group $\mathbb{Z}$.  
Then there exists an element $a \in \sum K^2$ with $v(a) = 1$.
\end{lemma}

\begin{proof}
Let $R$ be the valuation ring of $(K,v)$, $\mathfrak{m}$ the maximal ideal,
$\pi$ the uniformizer, and $k_v$ the residue field of $v$.
Since $\kappa_{v}$ is nonreal, there is a nontrivial representation 
$-1=\overline{x_1}^2 + \cdots + \overline{x_s}^2$ where each $x_i \in R$ 
and $\overline{x_i} \in \kappa_v$. 
Then $1 + x^2_1 +\ldots +x^2_s \in \mathfrak{m}$. 
First suppose that $\cha \kappa_v \ne 2$.
Then either $v(1 + x^2_1 +\cdots +x^2_s) = 1$ or 
$v((1+\pi)^2+ x^2_1 +\cdots +x^2_s ) = 1$. 

Now assume that $\cha \kappa_v = 2$.
Let $K_v$ denote the completion of $K$ with respect to the metric induced by $v$.
Then $K_v$ is a nonreal field because $\kappa_v$ is nonreal.
To see this, consider the equation $2^2 + 1^2 + 1^2 + 1^2 + 1^2 = 2^3$.
Suppose that $(2) = (\pi^e)$.  We can apply Hensel's lemma to the
quadratic form $f = x_1^2 + x_2^2 + x_3^2 + x_4^2 + x_5^2$ because
$f(2,1,1,1,1) \equiv 0 \bmod \pi^{3e}$,  
$f_{x_2}'(2,1,1,1,1) = 2 \not\equiv 0 \bmod \pi^{e+1}$, and $3e \ge 2e + 1$. 
(For example, see \cite[Chapter VI, Theorem 2.18]{La-05}.)
Thus $f$ is isotropic over $K_v$.  Then $K_v$ is a nonreal field and every element 
of $K_v$ is a sum of squares of elements from $K_v$.  In particular, $\pi$ is a sum
of squares of elements from $K_v$.  Since $K$ is dense in $K_v$, it follows that
$K$ contains an element $a \in \sum K^2$ such that $v(a) = 1$.
\end{proof}

\begin{lemma}\label{several nonreal valuations}
Let $K$ be a field of characteristic different from $2$, and let $v_1,\ldots,v_n$ be a finite number of distinct 
nonreal discrete valuations on $K$ each with value group $\mathbb{Z}$.
Then the induced homomorphism of groups
\[\tau: \left(\sum K^2\right)^\times \to \mathbb{Z}^n \text{ given by }
\sigma \mapsto (v_1(\sigma), \ldots, v_n(\sigma) )\]
is surjective.
\end{lemma}

\begin{proof}
It is sufficient to show that the canonical $\mathbb{Z}$-basis of $\mathbb{Z}^n$, 
$\{e_1, \ldots, e_n\}$, is contained in the image.
We first show that $(2\mathbb{Z})^n$ lies in the image.  
Let $(2b_1, \ldots, 2b_n) \in (2\mathbb{Z})^n$ be given.  By the 
weak approximation theorem for discrete valuations, there exists 
$a \in K^{\times}$ such that $(v_1(a), \ldots, v_n(a)) = (b_1, \ldots, b_n)$.  
Then $\tau(a^2)  = (2b_1, \ldots, 2b_n)$.

By \Cref{nonreal valuation}, there exists $y \in \sum K^2$ such that $v_1(y)=1$. 
Now choose $m_1,\ldots, m_n \in \mathbb{Z}$ such that 
$m_1 > \frac{1}{2} v_1(y)$ and $m_i < \frac{1}{2} v_i(y)$ 
for $ 2\leq i \leq n$.  

By weak approximation, there exists $z \in K$ such that $v_i(z) =m_i$ for 
$1 \leq i \leq n$.
Then $v_1(y + z^2) = v_1(y) = 1$ and $v_i(y+ z^2)=v_i(z^2) \in 2 \mathbb{Z}$ for all $2\leq i \leq n$.  
Then $\tau(y + z^2) \in e_1 + (2\mathbb{Z})^n$.  Since $(2\mathbb{Z})^n \in \im(\tau)$, it follows that $e_1 \in \im(\tau)$.  The same holds for each other
basis element $e_i$, $i \ge 2$.
\end{proof}

\noindent
Remark:
\Cref{nonreal valuation} and 
\Cref{several nonreal valuations} are variants of  \cite[Propositons 4.2 and 4.10]{BGV}.  Note that  \cite[Propositons 4.2 and 4.10]{BGV} do not consider the case when a residue field has characteristic $2$.

\begin{lemma}\label{complete splitting}
Let $K/E$ be a proper finite extension of degree $n$ of fields of characteristic different from $2$. Suppose there exists a nontrivial, nonreal discrete valuation $v$ on $E$ that extends to 
$n$ distinct valuations on $K$.  Then there exists 
$\sigma \in \left(\sum K^2\right) \backslash EK^2$.
\end{lemma}

\begin{proof}
Let $w_1, \ldots, w_n$ denote the $n$ distinct valuations on $K$ that extend $v$.
We may assume that the value group of each $w_i$ is equal to $\mathbb{Z}$. 
By \Cref{several nonreal valuations}, there exists 
$\sigma \in  \sum K^2$ such that $w_1(\sigma)$ is odd and $w_i(\sigma)$ is even for $i \geq 2$.
Suppose that $\left(\sum K^2\right) \subseteq EK^2$. 
Then we can find $\sigma$ with the same properties and such that $\sigma \in E$.  
But this yields a contradiction, since on the one hand 
$w_1(\sigma) = v(\sigma) = w_i(\sigma)$ for $i \geq 2$, 
and on the other hand $w_1(\sigma) \neq w_i(\sigma)$ for $i \ge 2$. 
\end{proof}

\noindent
Remark:  The proof above requires only that $v$ has at least two distinct
unramified extensions on $K$.  
In fact, P. Gupta observed that the two extensions do not even need to 
be unramified.
However, in view of our applications later, 
it is convenient to assume the stronger hypothesis that $v$ splits completely in $K$.

\medskip

\noindent
As already mentioned in the introduction, the following result shows that non-pythagorean fields are hereditarily non-pythagorean.
\begin{proposition}\label{Diller-Dress}
Let $E/F$ be a finite extension of fields. Assume that $E$ is pythagorean. Then $F$ is 
pythagorean.
\end{proposition}

\begin{proof}
This is trivially true in the case of characteristic $2$. See \cite[Chapter VIII, Theorem 5.7]{La-05} for the case of characteristic different from $2$.
\end{proof}

Let $L/F$ be any finite algebraic extension.
Then there is a unique field $K$ such that $F \subseteq K \subseteq L$ where
$K/F$ is separable and $L/K$ is purely inseparable.  
We use the convention that a field is purely inseparable over itself.

The following result gives a simpler criterion for a field to be 
hereditarily non-pythagorean

\begin{lemma}\label{hnq}
A field $k$ is hereditarily non-pythagorean if and only if every proper finite 
separable extension $F$ of $k$ is non-pythagorean. 
\end{lemma}

\begin{proof}
Assume that every proper finite separable extension of $k$ is non-pythagorean. 
Let $L$ be a proper finite extension of $k$ that is not purely inseparable over $k$.
There is a subfield $F$ of $L$ satisfying $k \subsetneq F \subseteq L$ such that
$F/k$ is a proper separable extension and $L/F$ is a purely inseparable extension.
By hypothesis, $F$ is non-pythagorean. 
By \Cref{Diller-Dress}, $L$ is non-pythagorean. 
Thus $k$ is hereditarily non-pythagorean.
\end{proof}

\begin{proposition}\label{test}
Let $F$ be a field.  Assume that for every proper finite separable extension 
$K/F$ and for each chain of fields $F \subseteq \tilde{F} \subsetneq K$, that
$\sum K^2 \not\subset \widetilde{F} K^2$.

Let $E_0/F$ be a (possibly infinite) Galois extension and assume that $E_0$ is
pythagorean. Then $E_0$ is hereditarily non-pythagorean.
\end{proposition}

\begin{proof}
Let $L$ be a proper finite separable extension of $E_0$.
Let $\beta \in L$ be a primitive element for $L/E_0$.  
The irreducible polynomial of $\beta$ over $E_0$ is defined over a finite extension 
$\widetilde{F}/F$ contained in $E_0$.  
Let $K=\widetilde{F}(\beta)$.  Then $\widetilde{F} \subsetneq K$.
Note that $L = E_0 K$ and that the fields $E_0$ and $K$ are linearily disjoint over 
$\widetilde{F}$. 
Since $E_0/\widetilde{F}$ is a Galois extension, it follows that $L/K$ is a Galois
extension.  We have an isomorphism $\mathop{Gal}(L/K) \to \mathop{Gal}(E_0/\widetilde{F})$ given by $\varphi \mapsto \varphi\vert_{E_0}$.

By hypothesis, there exists $\sigma \in \sum K^2$ with 
$\sigma \notin \widetilde{F}K^2$. In particular $\cha(F) \neq 2$.
We will show that $\sigma \notin L^2$. 
Suppose on the contrary that $K(\sqrt{\sigma}) \subseteq L$.
Then by Galois theory, there exists a quadratic extension 
$\widetilde{F}(\sqrt{\delta})/\widetilde{F}$ with $\delta \in \widetilde{F}$ 
such that $K(\sqrt{\delta}) = K(\sqrt{\sigma})$. 
This gives $\delta \sigma \in K^2$.
Then $\sigma \in \delta K^2 \subset \widetilde{F} K^2$,
which is a contradiction.  
Hence, $\sigma \in \sum K^2 \subseteq \sum L^2$, but $\sigma \notin L^2$.
Thus $L$ is not pythagorean and so $E_0$ is hereditarily non-pythagorean by \Cref{hnq}.
\end{proof}

\begin{proposition}\label{test2}
Let $F$ be a field of characteristic different from $2$, and assume that for each proper finite separable extension 
$K/F$ there is a nontrivial, nonreal discrete valuation $v$ on $F$ that splits completely in $K$.

Let $E_0/F$ be a (possibly infinite) Galois extension and assume that $E_0$ is pythagorean.
Then $E_0$ is hereditarily non-pythagorean.
\end{proposition}

\begin{proof}
Let $K/F$ be a proper finite separable extension and let
$F \subseteq \widetilde{F} \subsetneq K$ be a chain of fields.
The hypothesis implies that there is a nontrivial, nonreal discrete valuation $v$ 
on $\widetilde{F}$ that splits completely in $K$.
\Cref{complete splitting} implies that 
$\sum K^2 \not\subset \widetilde{F} K^2$.
The result now follows from \Cref{test}.
\end{proof}

\section{Pythagorean and quadratic closures of number fields}

In this section we deal with the case of number fields. 
The key ingredient for the main result (\Cref{main theorem}) is based on the following result.

\begin{proposition}\label{global sums of squares}
Let $K/E$ be a proper finite extension of number fields. Then \\
$\sum K^2 \not\subset EK^2$.
\end{proposition}

This result follows easily from \cite[Chapter VII, Theorem 7.12]{La-05}, 
which states (with an easily corrected typo) that 
$K^{\times}/E^{\times} (K^{\times})^2$ is an infinite group. 
Since $K$ has finitely many orderings $P_1, \ldots , P_n$ and
$P_1\cap \cdots \cap P_n =\sum K^2$, it follows that  
$K^\times/(\sum K^2)^\times$ injects into 
$(K^\times /P^\times_1) \times \cdots \times (K^\times/P^\times_n) $, and therefore $|K^\times/(\sum K^2)^\times | \le 2^n$. 
If $\sum K^2 \subset EK^2$,  then
$|K^{\times}/E^{\times} (K^{\times})^2| \le |K^{\times}/(\sum K^2)^{\times}| \le 2^n$, a contradiction.  Thus $\sum K^2 \not\subset EK^2$.

The proof of \cite[Chapter VII, Theorem 7.12]{La-05} is based on the number theoretic result that in a finite extension of number fields there exist infinitely many prime ideals that do not remain prime in the extension field.  We use this opportunity to give an alternative proof of \Cref{global sums of squares} which reduces to the local dyadic case, and where no result on the existence of splitting primes is needed.  Instead we require only the following result.

\begin{lemma}\label{dyadic square classes}
Let $E/\mathbb{Q}_2$ be a finite field extension and $K/E$ a proper finite extension.  Then $ E^\times (K^{\times})^2 \neq K^\times$.
\end{lemma}

\begin{proof}
Let $[E:\mathbb{Q}_2] = m$ and $[K:E] = n >1$, and thus 
$[K:\mathbb{Q}_2]= mn$.
By \cite[Chapter VI, Corollary 2.23]{La-05}, we have 
\[\left\vert E^\times /(E^{\times})^2  \right\vert = 2^{m+2} < 2^{mn + 2} 
= \left\vert K^\times / (K^{\times})^2  \right\vert.\]
This gives 
\[\left\vert E^\times (K^{\times})^2 / (K^{\times})^2 \right\vert 
= \left\vert E^{\times} / ( E^\times \cap (K^{\times})^2 )  \right\vert 
\leq \left\vert E^\times / (E^{\times})^2 \right\vert 
< \left\vert K^\times / (K^{\times})^2  \right\vert.\] \end{proof}

\begin{proof}[Alternative proof of \Cref{global sums of squares}:]
Let $K/E$ be a proper finite extension of number fields. 
Let $\mathfrak{p}$ be a dyadic prime ideal in the ring of integers of $E$.
Let $E_\mathfrak{p}$ be the completion of $E$ with respect to $\mathfrak{p}$. 
If the hypotheses of \Cref{complete splitting} hold, then 
$\sum K^2 \not\subset E K^2$.
If the hypotheses of \Cref{complete splitting} do not hold, then there exists 
a dyadic prime ideal $\mathcal{P}$ in the ring of integers $K$ such that 
$\mathfrak{p}=E \cap \mathcal{P}$ and 
$[ K_{\mathcal{P}}:E_{\mathfrak{p}}] > 1$, where $K_\mathcal{P}$ the completion of $K$ with respect to $\mathcal{P}$.
Then \Cref{dyadic square classes} implies that there exists
$\sigma_{\mathcal{P}} \in K_{\mathcal{P}}^{\times}\backslash 
E_{\mathfrak{p}}^{\times}(K_{\mathcal{P}}^{\times})^2$.

Since $K_{\mathcal{P}}$ is nonreal, we have  
$K_\mathcal{P}^\times = (\sum K_\mathcal{P}^2)^\times$. Thus
$\sigma_\mathcal{P} \in  (\sum K_\mathcal{P}^2)^\times \backslash 
E^\times_\mathfrak{p} (K_\mathcal{P}^{\times})^2$. 
Let $\sigma_\mathcal{P} = x_{\mathcal{P},1}^2 + \ldots + x_{\mathcal{P},m}^2$ for some $m \in \mathbb{N}$ and $x_{\mathcal{P},i} \in K_\mathcal{P}$.
For every $i \in \{1,\ldots,m\}$ let 
$\left(x^{(\ell)}_{\mathcal{P},i}\right)_{\ell \in \mathbb{N}}$ 
be a sequence in $K$ that converges to $x_{\mathcal{P},i}$ in the 
$\mathcal{P}$-adic metric.  
Let $\sigma^{(\ell)} = {x^{(\ell)\,\,2}_{\mathcal{P},1}} + \cdots + {x^{(\ell)\,\,2}_{\mathcal{P},m}}$. Then $\left(\sigma^{(\ell)}\right)_{\ell \in \mathbb{N}}$ 
is a sequence in $\sum K^2$ that converges to $\sigma_\mathcal{P}$ in the 
$\mathcal{P}$-adic metric.  
Since $\sigma_{\mathcal{P}} \ne 0$, for $\ell$ sufficiently large, 
$\left(\sigma^{(\ell)}\right)_{\ell \in \mathbb{N}}$ is a sequence in 
$\left(\sum K^2\right)^\times$. 
We finish the proof by showing that for $\ell$ sufficiently large, $\sigma^{(\ell)}$  is not contained $E_\mathfrak{p}^\times K_\mathcal{P}^{\times 2}$, which contains $E^\times K^{\times 2}$ as a subset.  This will follow easily after we show that
$E_\mathfrak{p}^\times K_\mathcal{P}^{\times 2}$ is a closed subset of 
$K_\mathcal{P}^\times$ in the $\mathcal{P}$-adic metric. 

So, let $(z_n)_{n \in \mathbb{N}}$ be a sequence in 
$E_\mathfrak{p}^\times K_\mathcal{P}^{\times 2}$ that converges to some 
$z \in K_\mathcal{P}^\times$. 
For every $n \in \mathbb{N}$ there exist $e_n \in E_\mathfrak{p}^\times$ and $x_n \in K_\mathcal{P}^{\times}$ such that $z_n= e_n x^2_n$.
After possibly multiplying $e_n$ with a square in $E_\mathfrak{p}^{\times 2}$ and $x^2_n$ with its inverse, if necessary, we can arrange that the sequence 
$(e_n)_{n\in \mathbb{N}}$ is bounded in the $\mathcal{P}$-adic metric, and since $(z_n)_{n\in \mathbb{N}}$ converges to a non-zero element and hence is also bounded,  we conclude that $x_n$ is also a bounded sequence. 
Since closed balls in the $\mathcal{P}$-adic metric are compact, every sequence in a closed ball has a convergent subsequence.  Then there exist subsequences 
$(e_{n_k})_{k\in \mathbb{N}}$ and $(x_{n_k})_{k\in \mathbb{N}}$ that converge to some element $e\in E_\mathfrak{p}$ and $x \in K_\mathcal{P}$, and thus $z_{n_k} =e_{n_k} x^2_{n_k}$ converges to $ex^2$ for $k \to \infty$. It follows that $z= e x^2$, and thus 
$z \in E_\mathfrak{p}^\times K_\mathcal{P}^{\times 2}$, showing that 
$E_\mathfrak{p}^\times K_\mathcal{P}^{\times 2}$ is a closed subset of 
$K_\mathcal{P}^{\times}$. \end{proof}

\begin{lemma}\label{compositum with E_0}
Let $K$ be a number field and let $E_0$ be an infinite dimensional number field 
that is pythagorean.
Let $\mathfrak{p}$ be a prime ideal of $K$ and $K_\mathfrak{p}$ the 
$\mathfrak{p}$-adic completion of $K$. Let $E_0 K_\mathfrak{p}$ be a 
compositum in some algebraic closure of $K_\mathfrak{p}$. 
Then $E_0 K_\mathfrak{p}$ is an infinite algebraic extension of
$K_\mathfrak{p}$.
\end{lemma}

\begin{proof}
Let $p = \mathbb{Q} \cap \mathfrak{p}$.
Consider the following two chains of fields.
\[\mathbb{Q}_p \subseteq E_0 \mathbb{Q}_p \subseteq E_0 K_{\mathfrak{p}}\]
\[\mathbb{Q}_p \subseteq K_{\mathfrak{p}} \subseteq E_0 K_{\mathfrak{p}}\]

We can use \Cref{nonreal valuation} to show that the
value group of any extension of the $p$-adic valuation to $E_0$ is 
$2$-divisible, but the extension of the $p$-adic valuation to $K_{\mathfrak{p}}$ 
is discrete.  
Therefore $[E_0 \mathbb{Q}_p:  \mathbb{Q}_p] = \infty$ and 
$[K_{\mathfrak{p}} : \mathbb{Q}_p]$ is finite.
It follows that $ [E_0 K_{\mathfrak{p}}:K_{\mathfrak{p}} ]= \infty$.
\end{proof}

\medskip

We now prove the main result of this section.

\begin{theorem}\label{main theorem}
Let $L$ be an algebraic extension of $\mathbb{Q}$ and assume that $L$ contains a pythagorean subfield $E_0$.
Then $p(L) \leq 2$.  
If moreover $1 < [L:E_0] < \infty$  and $E_0$ is a Galois extension of some number field, then $p(L) = 2$.  In particular, $E_0$ is hereditarily non-pythagorean.
\end{theorem}

\begin{proof} 
We prove first that $\sum L^2 = L^2 + L^2$.
Let $\sigma \in (\sum L^2)^{\times}$.  
Write $\sigma = \sum_{i=1}^m \alpha_i^2$, where each $\alpha_i \in L$.
 Let $K= \mathbb{Q}(\alpha_1, \ldots, \alpha_m)$. 
Any field composite $KE_0$ is an infinite field extension of $K$ contained in $L$, and $\sigma \in \sum K^2$.
We will show that there exists a finite Galois extension $\tilde{F}$ of $\mathbb{Q}$ 
contained in $E_0$ such that the quadratic form $q=\langle 1, 1, -\sigma \rangle$
is isotropic over $K\widetilde{F}$. 
This quadratic form is totally indefinite over $K\widetilde{F}$ for every such extension $\widetilde{F}/\mathbb{Q}$ because $q$ is totally indefinite over $K$. 
There are only finitely many prime ideals $\mathfrak{p}$ in $K$ such that 
$v_\mathfrak{p}(\sigma)\neq 0$.
Let $\mathcal{S}$ denote the set of prime ideals $\mathfrak{p}$ in the ring of integers of $K$ such that $v_\mathfrak{p}(\sigma)\neq 0$ together with the finitely many prime ideals containing $2$ (the dyadic primes).
By Springer's theorem, we have for any prime ideal 
$\mathfrak{p} \notin \mathcal{S}$ that $q$ is isotropic over the 
$\mathfrak{p}$-adic completion $K_\mathfrak{p}$, and the same is then true for 
$(K\widetilde{F})_\mathcal{P}$ for any finite Galois extension $\widetilde{F}/\mathbb{Q}$ and for any prime ideal $\mathcal{P}$ of $K\widetilde{F}$ extending $\mathfrak{p}$. 
Now let $\mathfrak{p} \in \mathcal{S}$. 
Since $\mathbb{Q}_\mathrm{pyth} \subseteq E_0$, we have by \Cref{compositum with E_0}, that a compositum 
$K_\mathfrak{p}\mathbb{Q}_\mathrm{pyth}$ is an infinite algebraic extension of $K_\mathfrak{p}$. 
There exists a finite Galois $2$-extension $\widetilde{F}/\mathbb{Q}$ with $\widetilde{F} \subseteq \mathbb{Q}_\mathrm{pyth}$ such that 
$K_\mathfrak{p}\widetilde{F}$ is a proper finite $2$-extension of $K_\mathfrak{p}$. Since $\mathcal{S}$ contains only finitely many prime ideals, one can find 
$\widetilde{F}$ such that this is the case for all primes 
$\mathfrak{p} \in \mathcal{S}$. 
For any prime $\mathcal{P}$ of $K\widetilde{F}$ lying over some prime 
$\mathfrak{p}$, we have that $(K\widetilde{F})_\mathcal{P}$ is 
$K_\mathfrak{p}$-isomorphic to $K_\mathfrak{p}\widetilde{F}$. 
If $\mathfrak{p} \notin S$, we already have that $q$ is isotropic over 
$K_\mathfrak{p}\widetilde{F}$. 
If $\mathfrak{p} \in S$ then
$(K\widetilde{F})_\mathcal{P}$ is a proper finite $2$-extension over 
$K_\mathfrak{p}$.  Then $q$ is isotropic over $(K\widetilde{F})_\mathcal{P}$ by
\cite[Chapter VI, Lemma 2.14]{La-05}. 
The Hasse Minkowski theorem implies that $q$ is isotropic over $K\widetilde{F}$, and hence over $L$. Thus $\sigma \in L^2 + L^2$.

Let us now assume that $1<[L:E_0]< \infty$ and  $E_0/F$ is a Galois extension for some (finite) number field $F$.  
Let $F \subseteq \tilde{F} \subsetneq K$ be a chain of fields with $[K:F] < \infty$.
By \Cref{global sums of squares}, 
$\sum K^2 \not\subset \widetilde{F} K^2$.
It now follows from \Cref{test} that $L^2 \subsetneq L^2 + L^2$.
\end{proof}

\noindent
Remark: For a field $E_0$ that is a quadratically closed and Galois over some number field, 
the conclusion in \Cref{main theorem} that $E_0$ is
hereditarily non-quadratically closed was already shown in \cite[Chapter VII, Corollary 7.11]{La-05}.

\smallskip

\noindent 
Note that if $L$ is an algebraic extension of $\mathbb{Q}$ such that $p(L)=2$, then, under the hypothesis that $L$ contains a pythagorean subfield, \Cref{main theorem} implies that $p(M)=2$ for all finite extensions $M/L$.
It turns out that the hypothesis of $L$ containing a pythagorean subfield is unnecessary. We will show this, using different techniques, in an upcoming paper.





\begin{corollary}\label{corollary of number fields}
Let $F$ be a number field.  Then $F_\mathrm{pyth}$ is hereditarily 
non-pythagorean and $F_\mathrm{quad}$ is hereditarily non-quadratically closed.
\end{corollary}

\begin{proof}
This result follows immediately from \Cref{main theorem} because
$F_{\mathrm{pyth}}/F$ and $F_{\mathrm{quad}}/F$ are both Galois extensions by \cite[Chapter VII, pp. 219-220]{La-05} or
by \cite[Chapter VIII, p. 258]{La-05}.
\end{proof}

\section{Inseparable extensions and pythagorean closures of fields}

We collect some results on separable and inseparable extensions 
with connections to pythagorean closures of fields that are needed in the 
following section.

All fields in this section have positive characteristic different from two. 
In particular, for such a field $E$, we have $E_\mathrm{pyth}= E_\mathrm{quad}$, and that $E$ is pythagorean if and only if $E$ is quadratically closed.

\begin{lemma}\label{inseparable}
Let $L/K$ be a finite purely inseparable algebraic extension of fields of characteristic different from $2$.
\begin{enumerate}
\item The canonical map of square classes 
$K^{\times}/(K^{\times})^{2} \to L^{\times}/(L^{\times})^{2}$
is an isomorphism.  In particular, $K$ is pythagorean if and only if $L$ is
pythagorean.
\item  $L_{\mathrm{pyth}}/K_{\mathrm{pyth}}$ is a finite purely inseparable
extension, and $L_{\mathrm{pyth}} = K_{\mathrm{pyth}} L$.
\end{enumerate}
\end{lemma}

\begin{proof}
(1) The map is injective because $[L:K]$ is odd.  Let $\alpha \in L^{\times}$. If $\cha(K)=p$ 
then $\alpha^{p^i} \in K^{\times}$ for some $i \ge 0$ because $L/K$ is 
purely inseparable.  Then 
$a: = \alpha^{p^i} = \alpha \alpha^{p^i - 1} \in \alpha (L^{\times})^2$.  
Thus the map is also surjective.

(2) It is clear that $K_{\mathrm{pyth}} L\subseteq L_{\mathrm{pyth}}$.  Let $\alpha \in L_{\mathrm{pyth}}^{\times}$.
Then $\alpha$ lies in some finite $2$-extension $N$ of $L$.  
Any $2$-extension of $K$ is a separable extension because $\cha K \ne 2$.  
Let $N_1/L$ be a quadratic extension.  It follows from (1) that $N_1 = L(\sqrt{a})$
for some $a \in K$.  
Then $N_1 = K(\sqrt{a})L$ and $N_1/K(\sqrt{a})$ is purely inseparable.
By repeating this for a tower of quadratic extensions that generates $N/L$, 
it follows from (1) that $N = LM$ where $M$ is some finite $2$-extension of $K$.
Since $M \subseteq K_{\mathrm{pyth}} $, we have that $\alpha \in N = LM \subseteq L K_{\mathrm{pyth}}$.
Therefore, $L_{\mathrm{pyth}} = K_{\mathrm{pyth}} L$ .  It follows that $L_{\mathrm{pyth}}/K_{\mathrm{pyth}}$  is purely inseparable because
$L/K$ is purely inseparable.
\end{proof}

\begin{lemma}\label{separable-inseparable}
Let $L/K$ be a finite purely inseparable algebraic extension of fields of characteristic different from $2$.
Let $N/L$ be a separable algebraic extension.  
Let $M$ be the maximal subfield of $N$ that is separable over $K$.
Then the following statements hold.
\begin{enumerate}
\item $N = LM$ and $N/M$ is purely inseparable.
\item $N/L$ is a Galois extension if and only if $M/K$ is a Galois extension.
\item $N$ is pythagorean if and only if $M$ is pythagorean.
\end{enumerate}
\end{lemma}

\begin{proof}

(1) Then $N/M$ is purely inseparable, and so 
$N/LM$ is both separable and purely inseparable. It follows that $N = LM$.
Since $L/K$ is purely inseparable, it follows that $N/M = LM/M$ is 
purely inseparable.

(2) Suppose that $M/K$ is a Galois extension.  
Let $\varphi:N \to L^{alg}$ be an $L$-isomorphism.  Then $\varphi$ is also a 
$K$-isomorphism. We have $\varphi(L) = L$ because $L/K$ is purely inseparable.
Since $M/K$ is a Galois extension, we have $\varphi(M) = M$. 
Thus $\varphi(N) = \varphi(LM) = \varphi(L)\varphi(M) = LM = N$ and so 
$N/L$ is a Galois extension.

Now suppose that $N/L$ is a Galois extension and let $\varphi: M \to K^{alg}$
be a $K$-isomorphism.   There exists an extension to a $K$-isomorphism 
$\varphi: N \to K^{alg}$.  Since $L/K$ is purely inseparable, it follows that
$\varphi: N \to K^{alg}$ is an $L$-isomorphism.  Since $N/L$ is a Galois extension,
we have $N = \varphi(N) = \varphi(L)\varphi(M) = L\varphi(M)$. 
Since $M/K$ is a separable extension, it follows that 
$\varphi(M)/K$ is a separable extension.  
Since $\varphi(M) \subset N$ and $M$ is the maximal subfield of $N$ 
that is separable over $K$, we have $\varphi(M) \subseteq M$.  
In a similar way, we have $M \subseteq \varphi(M)$, and thus $\varphi(M) = M$.  Therefore, $M/K$ is a Galois extension.

(3) Since $N/M$ is purely inseparable, the result follows from 
\Cref{inseparable} (1).
\end{proof}

\begin{proposition}\label{sep-purely-insep}
Let $L/K$ be a finite purely inseparable algebraic extension of fields of characteristic different from $2$.
Let $L_0/L$ be a separable extension and assume that $L_0$ is 
pythagorean.
Let $K_0$ be the maximal subfield of $L_0$ that is separable over $K$.

Then every purely inseparable extension of $K_0$ and $L_0$ is
pythagorean. 

Moreover, $K_0$ is hereditarily non-pythagorean if and only if
$L_0$ is hereditarily non-pythagorean.
\end{proposition}

\begin{proof}
By \Cref{separable-inseparable} (3), $K_0$ is pythagorean.
\Cref{inseparable} (1) implies that every purely inseparable extension of $K_0$
and $L_0$ is pythagorean.

Assume that $K_0$ is hereditarily non-pythagorean. Then  
every proper finite separable extension of $K_0$ is not pythagorean.  
Let $N$ be a proper finite separable extension of $L_0$. 
Since $L_0/K_0$ is purely inseparable by \Cref{separable-inseparable} (1), 
\Cref{separable-inseparable} 
implies that there exists a proper finite separable extension 
$M/K_0$ such that $N = L_0 M$.  
Since $M$ is not pythagorean by assumption and $N/M$ is purely 
inseparable, \Cref{inseparable} (1) implies that $N$ is not pythagorean.
Thus $L_0$ is hereditarily non-pythagorean by \Cref{hnq}.

Assume that $L_0$ is hereditarily non-pythagorean.
Then every proper finite separable extension of $L_0$ is not pythagorean.  Let $M$ be a proper finite separable extension of $K_0$.
Then $L_0 M$ is a proper finite separable extension of $L_0$.  Thus $L_0 M$ is
not pythagorean. Since $L_0 M/M$ is purely inseparable, 
\Cref{inseparable} (1) implies that $M$ is not pythagorean.
Thus $K_0$ is hereditarily non-pythagorean by \Cref{hnq}.
\end{proof}

\section{Pythagorean and quadratic closures of function fields}

The main result of this section, \Cref{main theorem II}, shows that the 
pythagorean closure of a field that is finitely generated of transcendence 
degree at least one over a subfield is 
hereditarily non-pythagorean. 
In order to be able to apply \Cref{test2} in the case of 
a general base field, we need to show a result on extensions of nonreal valuations, inspired by the statement from number theory that 
``there are infinitely many primes that split completely in a finite extension''.

\begin{theorem}\label{valuation splitting}
Let $k$ be field. Let $E/k$ be an algebraic function field in one variable and let 
$F/E$ be a finite separable extension.  Assume that $E$ is separable over some
rational function field of $k$.
Then there are infinitely many $k$-valuations $v$ 
on $E$ that extend to $[F:E]$ valuations on $F$.
\end{theorem}

\begin{proof}
By enlarging $F$ and taking a subfield of $E$ if necessary, 
it is sufficient to show the claim in the case where $F/E$ is a Galois extension and where $E=k(X)$ is a rational function field.
In this case, all $k$-valuations on $F$ extending a given $k$-valuation on $E$ are conjugate by the Galois group of $F/E$, and in order to show that there are $d:=[F:E]$ such extensions, it is equivalent to show that any one of them is unramified over $v$ and its residue field equals the residue field of $v$.

Let $f(T) \in E[T]$ be the minimal polynomial of a primitive element for $F/E$. By choosing the primitive element appropriately, we may assume that
the coefficients of $f$ lie in $k[X]$. Write 
$f=f(X,T) = T^d + \sum_{i=0}^{d-1} a_i(X) T^i$, where each $a_i(X) \in k[X]$.
Let $\Delta \in k[X]$ be the discriminant of $f$ considered as a polynomial in one variable over $k[X]$.
Any monic irreducible polynomial $p \in k[X]$ that does not divide $\Delta$ in $k[X]$ and that divides $f(X,g(X))$ for some $g(X) \in k[X]$ yields a $k$-valuation $v_p$ on $E=k(X)$ that extends to $[F:E]$ distinct $k$-valuations on $F$.
To see this, let $L = k[x]/(p(x))$.  
Then $L$ is a finite algebraic extension of $k$.  
Let $\alpha \in L$ be a root of $p$.
Let $v_p$ denote the valuation on $k(x)$ associated to the monic
irreducible polynomial $p$.
The element $g(\alpha) \in L$ is a root of the image of the polynomial $f(\alpha,T)$ 
in $k[x]/(p)$ because $p \mid f(X, g(X))$.
Thus the valuation $v_p$ extends to a valuation on $F$ that is unramified (because
$p \nmid \Delta$) with residue field equal to $L$ (because $g(\alpha) \in L$).
Since $F/E$ is Galois, it follows that $v_p$ extends to $d$ distinct valuations on $F$.

We denote by $\mathcal{P}_0$ the set of such $p$.
We further denote by $\mathcal{P}_\Delta$ the finite set of monic irreducible polynomials $p\in k[X]$ that divide $\Delta$.
Let $\mathcal{P}$ be any nonempty set of monic irreducible polynomials 
that contains $\mathcal{P}_0 \cup \mathcal{P}_\Delta$ and at most finitely many
other monic irreducible polynomials. 
We will show that $\mathcal{P}_0$ is infinite.
Suppose on the contrary that $\mathcal{P}_0$ is finite, and thus $\mathcal{P}$ is finite. Write $\mathcal{P}=\{p_1, \ldots, p_n\}$.
For any $r\in \mathbb{N}$, let
\[g_r(X) =  \prod_{j=1}^n p_j(X)^{r} \in k[X].\]
We will show that for $r$ sufficiently large, there exists a monic irreducible 
$q \in k[X]$, $q \notin \mathcal{P}$, such that $q$ divides
$f(X,g_r(X))$ in $k[X]$,
thereby yielding the contradiction that 
$q \in \mathcal{P}_0 \subseteq \mathcal{P}$.

For $r$ sufficiently large, we have that  
\[v_{p_j}(f(X,g_r(X))) = v_{p_j}(f(X, 0))\] 
for every $j \in \{1,\ldots,n\}$.  
Thus, for $r$ sufficiently large, the monic irreducible polynomials 
$\{p_1, \ldots, p_n\}$ appear always 
with the same multiplicity in the factorization of $f(X,g_r(X))$ in $k[X]$.
On the other hand, since $\mathcal{P}\neq \emptyset$, we have that for sufficiently large $r$, $\deg_X( f(X,g_r(X))) = rd\sum_{j=1}^n \deg_X(p_j )$, 
hence growing linearly in $r$. 
In particular, for all sufficiently large $r$, the polynomial $f(X,g_r(X)) \in k[X]$ has a monic irreducible factor $q$ distinct from $p_1,\ldots,p_n$. 
\end{proof}

\begin{corollary}\label{valuation splitting corollary}
Let $k$ be field. Let $E/k$ be an algebraic function field in one variable and let 
$F/E$ be a finite separable extension.  Assume that $E$ is separable over some
rational function field of $k$.
Then there are infinitely many 
nonreal 
$k$-valuations $v$ 
on $E$ that extend to $[F:E]$ valuations on $F$.
\end{corollary}

\begin{proof}
If $E$ is nonreal, then the statement is equivalent to \Cref{valuation splitting}. If $E$ is real, then we apply \Cref{valuation splitting} to the separable extension $F(\sqrt{-1})/E$.
Suppose that $v$ is a $k$-valuation on $E$ that splits completely in $F(\sqrt{-1})$.
Then any extended $k$-valuation is unramified over $E$ and has the same 
residue field as $v$.  Since $\sqrt{-1}$ is a unit in any 
valuation ring, the residue field contains $\sqrt{-1}$.  In particular, 
the residue field of $v$ is nonreal.
\end{proof}

By \cite[Chapter III, Proposition 9.2]{S}, if $k$ is perfect, then every function field
in one variable over $k$ is separable over some rational function field of $k$.

\begin{theorem}\label{main theorem II}
Let $E$ be a field of characteristic different from $2$ that is finitely generated of transcendence degree at least one over a subfield. 
Let $E_0/E$ be a Galois extension and assume that $E_0$  is either pythagorean or quadratically closed. Then $E_0$ is hereditarily non-pythagorean.
\end{theorem}

\begin{proof}
One can find a subfield $k\subset E$ such that $E$ is an algebraic function field in one variable over $k$.  
There is a subfield $E' \subseteq E$ such that $E'$ is a finite separable
extension of a rational function field of $k$ and $E$ is a finite purely 
inseparable extension of $E'$.  
Let $E'_0$ be the maximal subfield of $E_0$ that is separable over $E'$.
By \Cref{separable-inseparable}, $E_0 = E E'_0$, $E'_0/E'$ is 
Galois, and $E'_0, E_0$ are  both pythagorean or both quadratically closed.
By \Cref{valuation splitting corollary} and \Cref{test2},
$E'_0$ is hereditarily non-pythagorean.

Since $E/E'$ is purely inseparable, it follows that $E_0/E'_0$ is 
purely inseparable by \Cref{separable-inseparable} (1).
Thus $E_0$ is hereditarily non-pythagorean
by \Cref{sep-purely-insep}.
\end{proof}

\begin{corollary}\label{finitely generated transcendental}
Let $E$ be a field  of characteristic different from $2$ that is finitely generated of transcendence degree at least one over a subfield.  
Then $E_{\mathrm{pyth}}$ and $E_{\mathrm{quad}}$  are both hereditarily non-pythagorean.
\end{corollary}

\begin{proof}
As in \Cref{corollary of number fields}, the
result follows immediately from \Cref{main theorem II} because
$E_\mathrm{pyth}$ and $E_\mathrm{quad}$ are both Galois extensions of $E$.
\end{proof}

The following theorem gives a unified statement of our previous results in sections 3 and 5 for the case of a field that is finitely generated over its prime field. Recall that the term global field denotes either a number field or a function field in one variable over a finite field.

\begin{theorem}\label{unified} Let $E$ be a field  of characteristic different from $2$ that is finitely generated over its prime field. If $E$ contains a global field, then $E_{\mathrm{pyth}}$ and $E_{\mathrm{quad}}$  are both hereditarily non-pythagorean. Otherwise, $E$ is a finite field and $E_{\mathrm{pyth}}$ is hereditarily pythagorean.
\end{theorem}
\begin{proof}
If $\cha(E)=0$, then this follows immediately from \Cref{finitely generated transcendental} when $E$ has transcendence degree at least one over $\mathbb{Q}$,
and if it has transcendence degree $0$ over $\mathbb{Q}$ it follows from \Cref{corollary of number fields}.

Now assume $\cha(E)=p > 0$ for some $p \neq 2$. If $E$ has transcendence degree at least one over $\mathbb{F}_p$, the result follows from \Cref{finitely generated transcendental}. If $E$ is a finite field then $E_{\mathrm{pyth}}=E_{\mathrm{quad}}$ is hereditarily quadratically closed by \cite[Chapter VII, Theorem 7.15]{La-05}, since all finite fields have exactly two square classes.
\end{proof}

\noindent
Remark: In the statement of \Cref{unified}, we can replace ``finitely generated over its prime field'' by ``finite extension over any purely transcendental extension of its prime field''. We can also replace $E_{\mathrm{pyth}}$ or $E_{\mathrm{quad}}$ by an arbitrary Galois extension of $E$ that is pythagorean or quadratically closed, using directly \Cref{main theorem} and \Cref{main theorem II} in the proof.

\section{Some results for infinite algebraic extensions of number fields and function fields}

The next result extends \Cref{main theorem} and \Cref{main theorem II}.

\begin{theorem}\label{new}
Let $F$ be either a number field or a field of characteristic different from $2$ that is finitely generated of transcendence degree at least one over a subfield. 
Let $E/F$ be an infinite Galois extension.  
Then $E_{\mathrm{pyth}}$ and $E_{\mathrm{quad}}$  are both hereditarily non-pythagorean.
\end{theorem}

\begin{proof}
We first show that $E_{\mathrm{pyth}}$ and $E_{\mathrm{quad}}$  are Galois extensions of $F$. 
Let $\varphi \in \mathop{Gal}(F^{\mathrm{sep}}/F)$. 
We claim that $\varphi(E_{\mathrm{pyth}}) = E_{\mathrm{pyth}}$ and $\varphi(E_{\mathrm{quad}}) = E_{\mathrm{quad}}$.
It is straightforward to verify that $\varphi(E_{\mathrm{pyth}}) = \varphi(E)_\mathrm{pyth}$ and $\varphi(E_{\mathrm{quad}}) = \varphi(E)_\mathrm{quad}$ .  Since $E/F$ is a Galois extension, we have
that $\varphi(E) = E$.  Thus $\varphi(E_{\mathrm{pyth}})=E_\mathrm{pyth}$ and $\varphi(E_{\mathrm{quad}})=E_\mathrm{quad}$.

The result now follows from \Cref{main theorem} and \Cref{main theorem II}.
\end{proof}

\begin{example}\label{example}
Let $F$ be a field  of characteristic different from $2$ that is not separably algebraically closed  or real closed and let $F^\mathrm{sep}$ be its separable algebraic closure.  Then there exists
an automorphism $\varphi \in \mathop{Gal}(F^\mathrm{sep}/ F)$ 
having infinite order.
Let $E$ be the fixed field of $\varphi$.  Then $\mathop{Gal}(F^\mathrm{sep}/ E)$  is the pro-finite closure of the cyclic group generated by $\varphi$,
which is isomorphic to $\widehat{\mathbb{Z}}$, the inverse limit of all $ \Z / n\Z$.
Suppose that $E$ is a real field.  Then $E$ would be contained in a real closure.
However, $\widehat{\mathbb{Z}}$ does not contain an element of order two. 
Thus $E$ is a nonreal field.
Moreover $E$ is not quadratically closed, since $\widehat{\mathbb{Z}}$ contains a subgroup of index $2$.
$E_{\mathrm{quad}}$ is a Galois extension of $E$.
Then 
$$\mathop{Gal}(F^\mathrm{sep}/ E_{\mathrm{quad}} ) \cong 
\bigoplus_{\substack{\text{primes}\\ p \neq 2}} \Z_p$$
It follows that every finite extension of $E_{\mathrm{quad}} $ has odd degree. In particular, $E_{\mathrm{quad}} $ is hereditarily quadratically closed. 

If $F$ is either a number field or a field that is finitely generated of 
transcendence degree at least one over a subfield, 
then \Cref{new} implies that $E$ is not a Galois extension of $F$.
\end{example}

\begin{proposition}\label{infinite but perhaps not Galois}
Let $E/\mathbb{Q}$ be a (possibly infinite) algebraic extension and assume that 
$E$ is real but not pythagorean. 
Then no quadratic extension of $E_{pyth}$ is pythagorean.
\end{proposition}

\begin{proof}
Let $L=E_{pyth}(\sqrt{\alpha})$ be a quadratic extension of $E_{pyth}$ . 
Note that $\alpha$ is negative with respect to at least one ordering of $E_{pyth}$.
Since $E_{pyth}/E$ is an infinite totally real extension (every ordering of $E$ extends to $n$ orderings in every finite extension of degree $n$ in $E_{pyth}$), and since the element $\alpha$ is defined over a finite extension of $E$ in 
$E_{pyth}$, we conclude that $\alpha$ is negative at an infinite number of orderings.

Let $E'/\mathbb{Q}$ be a finite extension contained in $E_{pyth}$ such that 
$\alpha \in E'$ and such that  two of the orderings of $E_{pyth}$ with respect to which $\alpha$ is negative, restrict to two distinct orderings $<_1$ and $<_2$ on $E'$. Note that $L'=E'(\sqrt{\alpha})$ is a quadratic extension contained in $L$. 
 
By weak approximation in the number field $E'$, there exist $x,y \in E'$ such that for $\beta:=x^2 +  y^2\alpha$, we have $\beta >_1 0 $ and $\beta <_2 0$.
As a consequence, we have that both $\beta$ and $\beta \alpha$ are not totally positive in $E_{pyth}$. Hence 
\[\beta \notin (E_{pyth}^{\times})^2 \cup \alpha (E_{pyth}^{\times})^2 = 
E_{pyth} \cap L^2.\]
This shows that $\beta \notin L^2$.  Since $\beta = x^2 + y^2 \alpha$
and $\alpha \in L^2$, it follows that $\beta \in \sum L^2$.  
Hence $L$ is not pythagorean.
\end{proof}

\begin{theorem}\label{intermediate quadratic}
Let $E/\mathbb{Q}$ be a (possibly infinite) algebraic extension and assume 
that $E$ is real but not pythagorean.  
Let $L/E_{pyth}$ be a finite proper extension, and assume it can be written as a chain of 
field extensions in which one intermediate extension is a quadratic extension. 
Then $L$ is not pythagorean.
\end{theorem}

\begin{proof}
By \Cref{Diller-Dress}, we may assume that the quadratic extension is the 
last extension in the chain. Thus we can assume that 
$E_{pyth} \subseteq F_0 \subsetneq L$ where $L=F_0(\sqrt{\alpha})$ 
for some $\alpha \in F_0$. 

First suppose that $L$ is nonreal pythagorean (i.e. quadratically closed). 
Since $L/E_{pyth}$ is a finite extension and $E_{pyth}$
is a real field, it follows from \cite[Chapter VIII, Corollary 5.11]{La-05} 
that $E_{pyth}$ is euclidean.  
In particular, $E_{pyth}$ is uniquely ordered. 
However, since $E$ is not pythagorean, $E_{pyth}$ is an infinite extension by 
\Cref{Diller-Dress}, which is given by iteratively adjoining square roots of sums of squares, whereby $E_{pyth}$ has infinitely many orderings.  
This is a contradiction. 

Suppose now that $L$ is real pythagorean. 
Let $\widetilde{E}/ E$ be a finite extension in $E_{pyth}$ such that the 
minimal polynomial of some primitive element $\beta$ for 
$F_0/E_{pyth}$ is defined over $\widetilde{E}$ and such that 
$\alpha \in \widetilde{F}:= \tilde{E}(\beta)$.  
The second property is possible to arrange, since 
$F_0 = E_{pyth}(\beta)$ is the direct limit of $E'(\beta)$ where $E'$ runs 
through the finite extensions $E'$ inside $E_{pyth}$ of any $\widetilde{E}$ 
with the first property.

Since $F_0$ is also real pythagorean by \Cref{Diller-Dress}, 
we have that $F_0$ is the pythagorean closure of $\tilde{F}$ (which is not pythagorean).
We are now in the situation of \Cref{infinite but perhaps not Galois} where 
we only consider a quadratic extension of the pythagorean closure of a 
real number field that is not pythagorean.  We showed in 
\Cref{infinite but perhaps not Galois} that in this case $L$ is not pythagorean. Hence we have a contradiction also in this case.
\end{proof}

\noindent
In the next section, we prove some general results that imply the existence 
of real infinite number fields that are not pythagorean
and whose pythagorean closures are not hereditarily non-pythagorean.

\medskip

\section{Pythagorean closures admitting pythagorean finite extensions}

\begin{lemma}\label{first lemma}
Let $F/E$ be a finite extension of fields with $[F:E] = n$, $n$ odd.
The following statements are equivalent.
\begin{enumerate}
\item $\sum F^2 = \left(\sum E^2 \right) F^2$.
\item $\sum F^2 \subset EF^2$.
\item $N_{F/E}: (\sum F^2)^{\times} / (F^{\times})^2 \to (\sum E^2)^{\times} / (E^{\times})^2$ is injective.
\item $\alpha N_{F/E}(\alpha) \in F^2$ for all $\alpha \in \sum F^2$.
\end{enumerate}
\end{lemma}

\begin{proof} All statements are trivially true, and hence equivalent, when $\cha(E)=2$. We may therefore assume that $\cha(E)\neq 2$.

(1) $\Rightarrow$ (2).  This is trivial.

(2) $\Rightarrow$ (3). 
Assume that (2) holds and let $\alpha \in (\sum F^2)^{\times}$.
Then $\alpha = a \beta^2$ where $a \in E$ and $\beta \in F$.  
Assume that $N_{F/E}(\alpha) \in E^2$.
Then $N_{F/E}(\alpha) = a^n (N_{F/E}(\beta))^2 \in E^2$, which implies that
$a^n \in E^2$.  Since $n$ is odd, we have $a \in E^2$, and thus $\alpha \in F^2$.

(3) $\Rightarrow$ (4).
Assume that (3) holds.  Let $\alpha \in (\sum F^2)^{\times}$.
Then $N_{F/E}(\alpha N_{F/E}(\alpha) ) = N_{F/E}(\alpha)^{n+1} 
\in (E^{\times})^2$ because $n+1$ is even.  
Therefore $\alpha N_{F/E}(\alpha) \in F^2$.

(4) $\Rightarrow$ (1).
Assume that (4) holds and let $\alpha \in \sum F^2$.  
Then $N_{F/E}(\alpha) \in \sum E^2$ and 
$\alpha \in N_{F/E}(\alpha) F^2 \subset \left(\sum E^2 \right) F^2$.
Thus $\sum F^2 \subset \left( \sum E^2 \right) F^2$.  The other inclusion
is obvious.
\end{proof}

\begin{lemma}\label{second lemma}
Let $F/E$ be a finite extension of fields with $[F:E] = n$, $n$ odd.
Let $K = E(\sqrt{d})$ where $d \in \sum E^2$, and assume that $[K:E]=2$.
Let $L = KF = F(\sqrt{d})$.  Then $\sum F^2 \subset EF^2$ if and only if 
$\sum L^2 \subset KL^2$.
\end{lemma}

\begin{proof}
First assume that $\sum F^2 \subset EF^2$.
Let $\beta \in \sum L^2$.  
Let $\alpha = N_{L/F}(\beta) \in \sum F^2$.  Then
\begin{align*}
N_{L/F}(\beta N_{L/K}(\beta))& = N_{L/F}(\beta) N_{K/E}(N_{L/K}(\beta)) \\
&= N_{L/F}(\beta) N_{F/E}(N_{L/F}(\beta)) = \alpha N_{F/E}(\alpha) \in F^2,
\end{align*}
by \Cref{first lemma}.  It follows that $\beta N_{L/K}(\beta) \in FL^2$
by \cite[Chapter VII, Theorem 3.8]{La-05}.

Since $\beta \in \sum L^2$, we have 
$N_{L/K}(\beta) \in \sum K^2 \subset \sum L^2$. 
Since $d \in \sum E^2 \subset \sum F^2$ and $L = F(\sqrt{d})$, 
we have $F \cap \sum L^2 = \sum F^2$.  
These two observations give 
\[\beta N_{L/K}(\beta) \in \left(F \cap \sum L^2\right) L^2 = 
\left(\sum F^2\right) L^2 \subset EF^2 L^2 \subset K L^2.\]
Thus $\beta \in KL^2$ because $N_{L/K}(\beta) \in K$.
Therefore $\sum L^2 \subset KL^2$.

Now assume that $\sum L^2 \subset KL^2$.
Let $a \in \sum F^2$.  Then $a \in \sum L^2$ and \Cref{first lemma} implies 
that $aN_{L/K}(a) \in L^2$.   Since $N_{L/K}(a) = N_{F/E}(a)$, we have
$aN_{F/E}(a) \in L^2 \cap F = F^2 \cup dF^2 \subset EF^2$.
Then $a \in EF^2$ because $N_{F/E}(a) \in E$.  
Therefore, $\sum F^2 \subset EF^2$.
\end{proof}

\begin{proposition}\label{pythag-go-up}
Let $F/E$ be a finite extension of fields with $[F:E] = n$, $n$ odd.
Assume that $\sum F^2 \subset EF^2$.
Then $F_{pyth} = E_{pyth}F$ and $[F_{pyth}:E_{pyth}] = n$.
\end{proposition}

\begin{proof}
The statement is trivially true when $\cha(E)=2$. We may therefore assume that $\cha(E)\neq 2$.
Let $K/E$ be any finite extension where $K \subset E_{pyth}$.  Then 
then there exists a finite chain 
$E = E_0 \subset E_1 \subset \cdots \subset E_m = K$ where $E_j/E_{j-1}$ is
a quadratic extension with $E_j = E_{j-1}(\sqrt{d_j})$ and 
$d_j \in \sum E_{j-1}^2$, $1 \le j \le m$.
Let $L = KF$. Then $[L:K] = n$ and by induction, \Cref{second lemma} 
implies that $\sum L^2 \subset K L^2$.  \Cref{first lemma} implies that
$\sum L^2 = \left(\sum K^2 \right) L^2$.  

Since this holds for every finite extension $K/E$ where $K \subset E_{pyth}$, 
it follows that by setting $M = E_{pyth}F$ we have
$\sum M^2 = \left( \sum E_{pyth}^2 \right) M^2 = E_{pyth}^2 M^2 = M^2$.
Thus $M$ is pythagorean, and so $M = F_{pyth}$.
\end{proof}

\begin{proposition}\label{odd prime}
Let $p$ be an odd prime, and let $F/E$ be a finite separable extension of fields of characteristic different from $2$ with 
$[F:E] = p$.
\begin{enumerate}
\item There is an algebraic extension $E_1 / E$ such that with $F_1 = E_1 F$ the 
following statements hold.
\begin{enumerate}
\item $E_1$ and $F$ are linearly disjoint over $E$. In particular
$[F_1 : E_1] = [F:E] = p$.
\item $E_1$ and $F_1$ are pythagorean.
\end{enumerate}

\item Assume that $E$ is not pythagorean and let $a \in \sum E^2$, 
$a \notin E^2$.  
Assume that $F/E$ is a Galois extension.
There is an algebraic extension $E_2/E$ such that with $F_2 = E_2 F$ the
following statements hold.
\begin{enumerate}
\item $E_2$ and $F$ are linearly disjoint over $E$ so that 
$[F_2 : E_2] = [F:E] = p$.
\item $a \notin E_2^2$, and thus $a \notin F_2^2$, so that $E_2$ and $F_2 $
are each not pythagorean.
\item $\sum F_2^2 = F_2^2 \cup aF_2^2 \subset E_2 F_2^2$.
\end{enumerate}

\item If $F/E$ is a Galois extension and $E$ is formally real, 
then it can be arranged in (1) and (2) that $F_1$ and $F_2$ are each formally real.
\end{enumerate}
\end{proposition}

\begin{proof}


(1) 
Let $E_1$ be a maximal algebraic extension of $E$ such that $E_1$ and $F$ are
linearly disjoint over $E$.  Let $F_1 = E_1 F$. Then $[F_1 : E_1] = [F:E] = p$.
We will show that $F_1$ is pythagorean, and thus $E_1$ is also pythagorean.

Suppose that $\alpha \in \sum F_1^2$, $\alpha \notin F_1^2$.
Let $M$ be the Galois closure of $F_1(\sqrt{\alpha})/E_1$.
We now show that $p \mid [M:E_1]$ but $p^2 \nmid [M:E_1]$.
Let $L$ be the Galois closure of $F_1/E_1$.  
Then $[L:E_1] \mid p!$, and thus $p^2 \nmid [L:E_1]$.
Since $L/E_1$ is Galois and $[F_1(\sqrt{\alpha}) L: L] = 1 \text{ or } 2$, 
it follows that $[M:L]$ is a $2$-power, and so $p^2 \nmid [M:E_1]$.

Let $K$ be the fixed field of a Sylow $p$-subgroup of $\gal(M/E_1)$. 
Then $[M:K] = p$ and $p \nmid [K:E_1]$.  
It follows that $K$ and $F_1$ are linearly disjoint over $E_1$ and $M = KF_1$.
Then $K$ and $F$ are linearly disjoint over $E$.  The maximality of $E_1$ implies
that $K = E_1$.  Then $M = KF_1 = F_1$. 
Since $F_1 \subsetneq F_1(\sqrt{\alpha}) \subseteq M$, 
we obtain a contradiction, and thus $F_1$ is pythagorean.

(2) The proof of (2) is  similar to the proof of (1).  Let $E_2$ be a maximal 
algebraic extension of $E$ such that $E_2$ and $F$ are linearly disjoint over $E$
and such that $a \notin E_2^2$.  Let $F_2 = E_2 F$.  
Then $F_2/E_2$ is a Galois extension and $[F_2 : E_2] = [F:E] = p$.
Also, $E_2$ is not pythagorean, and thus $F_2$ is not pythagorean.
We now show that $\sum F_2^2 = F_2^2 \cup a F_2^2$.
Let $\beta \in \sum F_2^2$, $\beta \notin F_2^2$.  

Assume first that $N_{F_2/E_2}(\beta) \in E_2^2$.
Let $M$ be the Galois closure of $F_2(\sqrt{\beta})/E_2$.
We show as in (1) that 
$p \mid [M:E_2]$ but $p^2 \nmid [M:E_2]$.
Let $K$ be the fixed field of a Sylow $p$-subgroup of $\gal(M/E_2)$.
Then $[M:K]=p$ and $p \nmid [K:E_2]$.  Then $K$ and $F_2$ are linearly 
disjoint over $E_2$, and thus $K$ and $F$ are linearly disjoint over $E$.  
If $a \notin K^2$, then the maximality of $E_2$ implies that $K = E_2$, which
leads to a contradiction, as in (1).  Thus $a \in K^2 \subset M^2$.
Since $a \notin E_2^2$ and $[F_2:E_2]$ is odd, we have $a \notin F_2^2$.  
Let $\gal(F_2/E_2) = \{\sigma_1, \ldots, \sigma_p\}$.
Since $F_2/E_2$ is Galois,
$M = F_2(\sqrt{\sigma_1(\beta)}, \ldots, \sqrt{\sigma_p(\beta)})$ and
$a \in F_2 \cap M^2 = \cup \, \sigma_1(\beta)^{e_1} \cdots
\sigma_p(\beta)^{e_p} F_2^2$, where $e_1, \ldots, e_p \in \{0,1\}$.
This is a contradiction because $N_{F_2/E_2}(a) \in a E_2^2 \ne E_2^2$ but
$N_{F_2/E_2}(\sigma_1(\beta)^{e_1} \cdots \sigma_p(\beta)^{e_p}) \in E_2^2$.

It follows that $\sum F_2^2 \subset E_2 F_2^2$ by 
\Cref{first lemma}, (3) $\Rightarrow$ (2).  Thus we may assume that
$\beta \in E_2$.  Then $M = F_2(\sqrt{\beta})$ and so $K = E_2(\sqrt{\beta})$.
Since $a \in E_2 \cap K^2 = E_2^2 \cup \beta E_2^2$, it follows that 
$\beta \in a E_2^2$.  Thus $\beta \in a E_2^2 \subset a F_2^2$.
Therefore $\sum F_2^2 = F_2^2 \cup a F_2^2$.

(3) Suppose that $F/E$ is a Galois extension and that $E$ is formally real, and thus
$F$ is formally real. 
Let $R$ be a real closure of $E$.   
Modify the proofs of (1) and (2) by choosing $E_1$ and $E_2$ maximal algebraic extensions of $E$ \textit{contained in $R$} that are linearly disjoint from $F$ 
over $E$, and for $E_2$ additionally require that $a \notin E_2^2$.  
Since $F/E$ is Galois, we
have that $F \subset R$, and thus it follows that $M$, the Galois closure of $F_1(\sqrt{\alpha})/E_1$ or $F_2(\sqrt{\beta})/E_2$, also lies in $R$.
The rest of the proof of (3) now follows the proof of (1) or (2).
\end{proof}

\begin{corollary}\label{summary}
Let $p$ be an odd prime and let $E$ be any non-pythagorean field that admits 
a Galois extension $F/E$ with $[F:E] = p$.  
\begin{enumerate}
\item Then there exists a non-pythagorean algebraic extension $E_0/E$ such that $E_0$ and $F$ are linearly disjoint over $E$ and setting $F_0 = E_0 F$, we have
$(F_0)_{pyth} = (E_0)_{pyth} F$ and $[(F_0)_{pyth}: (E_0)_{pyth}] = [F:E] = p$.
\item If $E$ is real, then it can be arranged that $E_0$ is real in addition to the
conditions in (1).
\end{enumerate}
In particular, $(E_0)_{pyth}$ is not hereditarily non-pythagorean.
\end{corollary}

\begin{proof}
By \Cref{odd prime} (2), there exists an algebraic extension $E_0/E$ such that with
$F_0 = E_0 F$, we have $E_0$ and $F$ are linearly disjoint over $E$, 
$[F_0:E_0] = p$, $E_0$ is non-pythagorean, and 
$\sum (F_0)^2 \subset E_0 F_0^2$.
If $E$ is real, then by \Cref{odd prime}, we can arrange that $F_0$ is real.
\Cref{pythag-go-up} implies that 
$(F_0)_{pyth} = (E_0)_{pyth} F$ and $[(F_0)_{pyth}: (E_0)_{pyth}] = 
[F_0:E_0] = p$.
\end{proof}

\begin{example}
Let $p$ be an odd prime, and let $E$ be any number field.  There exists a Galois extension $F/E$ of degree $p$  (such an extension can be found inside an abelian extension of $E$ given by adjoining a primitive $p^n$-th root of unity for sufficiently large $n$ such that $p$ divides the degree of the abelian extension).   Let $E_0/E$ be a non-pythagorean algebraic extension as in \Cref{summary}. Thus $(E_0)_\mathrm{pyth}$ is not hereditarily non-pythagorean. In particular, $E_0/E$ is infinite dimensional by \Cref{corollary of number fields} and $E_0$ is not Galois over any number field, by \Cref{new}.
\end{example}

\noindent
Remark: 
\Cref{odd prime} (2), (3), as well as \Cref{summary} can be extended to the situation where the Galois extension F/E has any odd degree (not necessarily prime).  The Feit-Thomson theorem (\cite{F-T}) implies that the Galois group of 
$F/E$ is solvable. 
Since $M/F_2$ is a (solvable) $2$-extension and $F_2/E_2$ is a solvable extension,
it follows that $M/E_2$ is a solvable Galois extension.  
Since $[M:F_2]$ is a $2$-power, $n$ is odd, and $\gal(M/E_2)$ is a solvable 
group, a theorem of P. Hall (\cite{H}, p. 141) implies that $\gal(M/E_2)$ has a subgroup of order $n$.  We let $K$ be the subfield of $M$ 
corresponding to this subgroup of order $n$.  The rest of the proof of 
\Cref{odd prime} (2), (3) and \Cref{summary} proceeds as above.  

\medskip
\noindent
Acknowledgements:  We thank Karim Becher, Eberhard Becker, and Parul Gupta
for useful conversations on these topics. 
 Financial support by CONICYT (proyecto FONDECYT 11150956) and by the Universidad de Santiago (USACH. proyecto DICYT, Codigo 041933G) are gratefully acknowledged.

\end{document}